\documentclass[graybox]{svmult}

\usepackage{amssymb}
\usepackage{eurosym}
\usepackage[T1]{fontenc}
\usepackage{times}
\usepackage{mathptmx}       
\usepackage{helvet}         
\usepackage{courier}        
\usepackage{makeidx}         
\usepackage{graphicx}        
\usepackage{multicol}        
\usepackage[bottom]{footmisc}

\newtheorem{ex}{Example}
\newtheorem{prb}{Problem}
\newtheorem{prop}{Proposition}

\newtheorem{rem}{Remark}

\newtheorem{thm}{Theorem}

\newcommand\ind{1\hspace{-2.2mm}1}
\newcommand{\R}{\mathbb{R}}

\newcommand\N{\mathbb N}

\newcommand{\Z}{\mathbb{Z}}

\newcommand\dju[4]{\bigcup_{#1}^{#2}\hspace{-#3mm\cdot}\hspace{#4mm}}

\newcommand\rsi{[\![}
\newcommand\lsi{]\!]}
\newcommand\csbl{[\hspace{-1.2mm}[} 
\newcommand\csbr{\,]\hspace{-1.2mm}]} 

\newcommand{\bthe}{\begin{Theorem}}
\newcommand{\ethe}{\end{Theorem}}

\newcommand{\dotcup}{\ensuremath{\mathaccent\cdot\cup}}
\newcommand\comment[1]{}

\makeindex             

\begin{document}

\title*{On Jump Measures of Optional Processes with Regulated Trajectories}
\author{Frank Oertel}
\institute{Frank Oertel \at Deloitte \& Touche GmbH, FSI Assurance - Quantitative Services \& Valuation, Rosenheimer Platz 4, D-81669 Munich\\
\email{f.oertel@email.de}
}
%
%
\maketitle
\abstract*{Starting from an iterative and hence numerically easily implementable representation of the thin set of jumps of a c\`{a}dl\`{a}g adapted stochastic process $X$ (including a few applications to the integration with respect to the jump measure of $X$), we develop similar representation techniques to describe the set of jumps of optional processes with regulated trajectories and introduce their induced jump measures with a view towards the framework of enlarged filtration in financial mathematics.}
\abstract{Starting from an iterative and hence numerically easily implementable representation of the thin set of jumps of a c\`{a}dl\`{a}g adapted stochastic process $X$ (including a few applications to the integration with respect to the jump measure of $X$), we develop similar representation techniques to describe the set of jumps of optional processes with regulated trajectories and introduce their induced jump measures with a view towards the framework of enlarged filtration in financial mathematics.}
\section{Preliminaries and Notation}
\label{sec:1}
In this section, we introduce the basic notation and terminology
which we will use throughout in this paper. 
Most of our notation and definitions including those ones
originating from the general theory of stochastic processes and
stochastic analysis are standard. We refer the reader to the
monographs \cite{CW1990}, \cite{HWY1992}, \cite{JS2003} 
and \cite{P2004}.

Since at most countable unions of pairwise disjoint sets play an important
role in this paper, we use a well-known symbolic abbreviation. For example, if
$A : = \bigcup_{n=1}^{\infty} A_n$, where $(A_n)_{n \in \N}$ is a
sequence of sets such that $A_i \cap A_j = \emptyset$ for all $i
\not= j$, we write shortly $A : = \dju{n=1}{\infty}{7.0}{6.0}A_n$. 

Throughout this paper, $(\Omega , \mathcal{F}, {\bf{F}}, {\mathbb{P}})$
denotes a fixed probability space, together with a fixed filtration
${\bf{F}}$. Even if it is not explicitly emphasized, the filtration
${\bf{F}} = (\mathcal{F}_t)_{t \geq 0}$ always is supposed to
satisfy the usual conditions\footnote{$\mathcal{F}_0$ contains all
$\mathbb{P}$-null sets and ${\bf{F}}$ is right-continuous.}.
A real-valued (stochastic) process $X : \Omega \times \R^+
\longrightarrow \R$ (which may be identified with the family of
random variables $(X_t)_{t \geq 0}$, where $X_t(\omega) : =
X(\omega, t)$)\footnote{$\R^{+} : = [0, \infty)$.} is called \textit{adapted} (with respect to
${\bf{F}}$) if $X_t$ is ${\mathcal{F}_t}$-measurable for all $t \in \R^+$. $X$
is called \textit{right-continuous} (respectively
\textit{left-continuous}) if for all $\omega \in \Omega$ the
trajectory $X_{\bullet}(\omega) : \R^+ \longrightarrow \R, t \mapsto
X_t(\omega)$ is a right-continuous (respectively left-continuous)
real-valued function. If all trajectories of $X$ do have left-hand
limits (respectively right-hand limits) everywhere on $\R^+$, $X^{-}
= (X_{t -})_{t \geq 0}$ (respectively $X^{+} = (X_{t +})_{t \geq
0})$ denotes the \textit{left-hand} (respectively
\textit{right-hand}) \textit{limit process}, where $X_{0 -} : =
X_{0+}$ by convention. If all trajectories of $X$ do have left-hand
limits and right-hand limits everywhere on $\R^+$, the \textit{jump
process} $\Delta X = (\Delta X_t)_{t \geq 0}$ is well-defined on
$\Omega \times \R^+$. It is given by $\Delta X : = X^{+} - X^{-}$.
 
A right-continuous process whose trajectories do have left limits 
everywhere on $\R^+$, is known as a
\textit{c\`{a}dl\`{a}g} process. If $X$ is $\mathcal{F} \otimes
\mathcal{B}({\mathbb{R}}^+)$-measurable, $X$ is said to be
\textit{measurable}. $X$ is said to be \textit{progressively
measurable} (or simply \textit{progressive}) if for each $t \geq 0$,
its restriction $X\vert_{\Omega \times [0, t]}$ is $\mathcal{F}_t
\otimes \mathcal{B}([0, t])$-measurable. Obviously, every
progressive process is measurable and (thanks to Fubini) adapted.

A random variable $T : \Omega \longrightarrow [0, \infty]$ is said
to be a \textit{stopping time} or \textit{optional time} (with
respect to ${\bf{F}}$) if for each $t \geq 0$, $\{T \leq t\} \in
{\mathcal{F}}_t$. Let $\mathcal{T}$ denote the set of all stopping
times, and let $S, T \in \mathcal{T}$ such that $S \leq T$. Then
$\rsi S, T \rsi : = \{ (\omega , t) \in \Omega \times \R^+ :
S(\omega) \leq t < T(\omega)\}$ is an example for a
\textit{stochastic interval}. Similarly, one defines the stochastic
intervals $\lsi S, T \lsi$, $\lsi S, T \rsi$ and $\rsi S, T \lsi$.
Note again that $\rsi T \lsi : = \rsi T, T \lsi =
\textup{Gr}(T)\vert_{\Omega \times \R^+}$ is simply the graph of the
stopping time $T : \Omega \longrightarrow [0, \infty]$ restricted to
$\Omega \times \R^+$. $\mathcal{O} = \sigma\big\{[\![T,\infty [\![
\hspace{1mm}: T \in \mathcal{T}\big\}$ denotes the \textit{optional
$\sigma$-field} which is generated by all c\`{a}dl\`{a}g adapted
processes. The \textit{predictable $\sigma$-field} $\mathcal{P}$ is
generated by all left-continuous adapted processes. An
$\mathcal{O}$- (respectively $\mathcal{P}$-) measurable process is
called \textit{optional} or \textit{well-measurable} (respectively
\textit{predictable}).
All optional or predictable processes are adapted. 

For the convenience of the reader, we recall and summarise the precise
relation between those different types of processes in the following
\begin{theorem}\label{thm:POPA}
Let $(\Omega , \mathcal{F}, {\bf{F}}, {\mathbb{P}})$ be a filtered
probability space such that ${\bf{F}}$ satisfies the usual
conditions. Let $X$ be a (real-valued) stochastic process on $\Omega \times
{\mathbb{R}}^+$. Consider the following statements:
\begin{description}[(iii)]
\item[(i)] $X$ is predictable;
\item[(ii)] $X$ is optional;
\item[(iii)] $X$ is progressive;
\item[(iv)] $X$ is adapted.
\end{description}
Then the following implications hold:
\[
\textstyle{(i)} \Rightarrow \textstyle{(ii)} \Rightarrow
\textstyle{(iii)} \Rightarrow \textstyle{(iv)}.
\]
If $X$ is right-continuous, then the following implications hold:
\[
\textstyle{(i)} \Rightarrow \textstyle{(ii)} \iff \textstyle{(iii)}
\iff \textstyle{(iv)}.
\]
If $X$ is left-continuous, then all statements are equivalent.
\end{theorem}
\begin{proof}
\smartqed
The general chain of implications $\textstyle{(i)} \Rightarrow
\textstyle{(ii)} \Rightarrow \textstyle{(iii)} \Rightarrow
\textstyle{(iv)}$ is well-known (for a detailed discussion cf.
e.\,g. \cite[Chapter 3]{CW1990}). If $X$ is left-continuous and
adapted, then $X$ is predictable. Hence, in this case, all four
statements are equivalent. If $X$ is right-continuous and adapted,
then $X$ is optional (cf. e.\,g. \cite[Theorem 4.32]{HWY1992}). In particular, $X$ is
progressive. \qed
\end{proof}
Recall that a function $f : \R^{+} \longrightarrow \R$ is said to be \textit{regulated on $\R^{+}$} if $f$ has right- and left-limits everywhere on $(0, \infty)$ and $f(0+)$ exists (cf. e.\,g. \cite[Ch. 7.6]{D1960}). 

Let us also commemorate the following 
\begin{lemma}\label{lemma:optional and regulated paths}
Let $X : \Omega \times \R^+ \longrightarrow \R$ be a stochastic process such that its trajectories are regulated. Then all trajectories of the left limit process $X^{-}$ $($respectively of the right limit process $X^{+}$$)$ are left-continuous $($respectively right-continuous$)$. If in addition $X$ is optional, then $X^{-}$ is predictable and $X^+$ is adapted.
\end{lemma}

Given an optional process $X$ with regulated trajectories, we put  
\[
\{\Delta X \not= 0 \} : = \{(\omega, t) \in \Omega \times {{\R}^+} : \Delta X_t(\omega) \not= 0\}\,.
\]
Recall the important fact that for any $\varepsilon > 0$ and any regulated function $f : \R^+ \longrightarrow \R$ the set $J_f(\varepsilon) := \{t >0 : \vert \Delta f(t) \vert > \varepsilon\}$ is at most countable, implying that
\[
J_f : = \{t > 0: \Delta f(t) \not= 0\} = \{t > 0 : \vert
\Delta f(t) \vert > 0\} = \bigcup\limits_{n \in \N}J_f(\frac{1}{n})
\]
is at most countable as well (cf. \cite[p. 286-288]{Ho1921} and \cite[Theorem 1.3]{K2004}). 
\section{Construction of Thin Sets of Jumps of C\`{a}dl\`{a}g Adapted Processes}
\label{sec:2}
In the general framework of semimartingales with jumps (such as e.\,g. L\'{e}vy processes) there are several ways to describe a stochastic integral with respect to a (random) jump measure $j_X$ of a c\`{a}dl\`{a}g adapted stochastic process $X = (X_t)_{t\geq 0}$. One approach is to implement the important subclass of ``thin'' subsets of ${\Omega \times \mathbb{R}}^{+}$ (cf. \cite[Def. 1.30]{JS2003}) in order to analyse the set $\{\Delta X \not= 0 \}$: 
\begin{theorem}[Dellacherie, 1972]\label{thm:Dellacherie}
Let $X = (X_t)_{t \geq 0}$ be an arbitrary ${\bf{F}}$-adapted c\`{a}dl\`{a}g stochastic process on $(\Omega , {\mathcal{F}}, {\bf{F}}, {\mathbb{P}})$. Then there exist a sequence $(T_n)_{n \in {\mathbb{N}}}$ of ${\bf{F}}$-stopping times such that $\rsi T_n \lsi \cap \rsi T_k \lsi  = \emptyset$ for all $n \not= k$ and
\[
\{\Delta X \not= 0 \} = \dju{n=1}{\infty}{3.3}{2.0}\rsi T_n \lsi\,. 
\]
In particular, $\Delta X_{T_n(\omega)}(\omega) \not= 0$ for all $\omega \in \Omega$ and $n \in {\mathbb{N}}$.   
\end{theorem}
A naturally appearing, iterative and hence implementable exhausting representation is given in the following important special case (cf. e.\,g. \cite[p. 25]{P2004} or the proof of \cite[Lemma 2.3.4.]{A2009}):
\begin{prop}\label{thm:iterative exhausting representation}
Let $X = (X_t)_{t \geq 0}$ be an arbitrary ${\bf{F}}$-adapted c\`{a}dl\`{a}g stochastic process on $(\Omega , {\mathcal{F}}, {\bf{F}}, {\mathbb{P}})$ and $A \in {\mathcal{B}}({\mathbb{R}})$ such that $0 \notin \overline{A}$. Put
\[
T_1^A(\omega) : = \inf\{t > 0 : \Delta X_t(\omega) \in A\}
\]
and
\[
T_n^A(\omega) : = \inf\{t > T_{n-1}^A(\omega) : \Delta X_t(\omega) \in A\} \hspace{5mm} (n \geq 2).
\]
Up to an evanescent set $(T_n^A)_{n \in {\mathbb{N}}}$ defines a sequence of strictly increasing ${\bf{F}}$-stopping times, satisfying
\[
\{\Delta X \in A \} = \dju{n=1}{\infty}{3.3}{2.4}\rsi S_n^A \lsi \,, 
\]
where
\[
S_n^A : = T_n^A \,\ind_{A}\big(\Delta X_{T_n^A}\big) + (+\infty) \,\ind_{A^c}\big(\Delta X_{T_n^A}\big) \,. 
\]
\end{prop}
\begin{proof}
\smartqed
In virtue of \cite[Chapter 4, p. 25ff]{P2004} each $T_n^A$ is a ${\bf{F}}$-stopping time and $\Omega_0 \times {\mathbb{R}}^{+}$ is an evanescent set, where $\Omega_0 : = \{\omega \in \Omega : \lim\limits_{n \to \infty} T_n^A(\omega) < \infty\}$. 
Fix $(\omega, t) \notin \Omega_0 \times {\mathbb{R}}^{+}$.
{}
Assume by contradiction that $T_{m_0}^A(\omega) = T_{m_0 + 1}^A(\omega) = : t^\ast$ for some $m_0 \in \mathbb{N}$. By definition of $t^\ast = T_{m_0 + 1}^A(\omega)$, there exists a sequence $(t_n)_{n \in \mathbb{N}}$ such that for all $n \in \mathbb{N}$ $\lim\limits_{n \to \infty} t_n = t^\ast$, $\Delta X_{t_n}(\omega) \in A$, and $t^\ast = T_{m_0}^A(\omega) < t_{n+1} \leq t_n$. Consequently, since $X$ has right-continuous paths, it follows that $\Delta X_{t^\ast}(\omega) = \lim\limits_{n \to \infty} \Delta X_{t_n}(\omega) \in \bar{A}$, implying that $\Delta X_{t^\ast}(\omega) \not= 0$ (since $0 \notin \overline{A}$). Thus $\lim\limits_{n \to \infty} t_n = t^\ast$ is an accumulation point of the at most countable set $\{t > 0 : \Delta X_t(\omega) \not=0 \}$ - a contradiction. 

To prove the set equality let firstly $\Delta X_t(\omega) \in A$. Assume by contradiction that for all $m \in {\mathbb{N}}$ $T_m^A(\omega) \not= t$. Since $\omega \notin \Omega_0$, there is some $m_0 \in {\mathbb{N}} \cap [2, \infty)$ such that $T_{m_0}^A(\omega) > t$. Choose $m_0$ small enough, so that $T_{m_0 - 1}^A(\omega) \leq t < T_{m_0}^A(\omega)$. Consequently, since $\Delta X_t(\omega) \in A$, we must have $t \leq T_{m_0 - 1}^A(\omega)$ and hence $T_{m_0 - 1}^A(\omega) = t$. However, the latter contradicts our assumption. Thus, $\{\Delta X \in A \} \subseteq  \bigcup_{n=1}^{\infty} \rsi T_n^A \lsi \,$.
{}
The claim now follows from \cite[Theorem 3.19]{HWY1992}.\qed
\end{proof}
\begin{rem}
Note that $\{S_n^A < +\infty \} \subseteq \{\Delta X_{T_n^A} \in A\} \subseteq \{S_n^A = T_n^A\}$. Hence,
\[
\ind_{A}\big(\Delta X_{T_n^A}\big)\ind_{\{T_n^A \leq t\}} = \ind_{\{S_n^A \leq t\}}
\]
for all $n \in {\mathbb{N}}$.
\end{rem}

Next, we recall and rewrite equivalently the construction of a random measure on ${\mathcal{B}}{}({\mathbb{R}}^{+} \times {\mathbb{R}}{})$ (cf. e.\,g. \cite[Def. 1.3]{JS2003}):
\begin{definition}
A random measure on $\R^{+} \times \R$ is a family $\mu \equiv (\mu(\omega; d(s, x)) : \omega \in \Omega)$ of non-negative measures on $(\R^{+} \times \R, {\mathcal{B}}{}({\mathbb{R}}^{+} \times {\mathbb{R}}{}))$, satisfying $\mu(\omega; \{0\} \times \R) = 0$ for all $\omega \in \Omega$.   
\end{definition}
Given an adapted $\R$-valued c\`{a}dl\`{a}g process $X$, a particular (integer-valued) random measure (cf. e.\,g. \cite[Prop. 1.16]{JS2003}) is given by the \textit{jump measure of $X$}, defined as
\begin{eqnarray*}
j_X(\omega, B) & : = & \sum_{s
>0}\ind_{\{\Delta X \not= 0\}}(\omega, s)\,\varepsilon_{\big(s, \Delta
X_{s}(\omega)\big)}(B)\\
& = & \sum_{s > 0} \ind_B\big(s, \Delta X_s(\omega)\big)\ind_{{\mathbb{R}}^\ast}(\Delta X_s(\omega))\\
& = & \#\big\{s > 0 : \Delta X_s(\omega) \not= 0 \mbox{ and } \big(s, \Delta X_s(\omega)\big) \in B \big\}\,,
\end{eqnarray*}
where $\varepsilon_{a}$ denotes the Dirac measure at point $a$ and $B \in {\mathcal{B}}{}({\mathbb{R}}^{+} \times {\mathbb{R}}{})$.

\comment{
Implementing the exhausting series of stopping times $(T_n)_{n \in \N}$ of the thin set $\{\Delta X \not= 0\}$ from Theorem \ref{thm:Dellacherie}, we immediately obtain 
\begin{corollary}
Let $B \in {\mathcal{B}}{}({\mathbb{R}}^{+} \times {\mathbb{R}}{})$ and $\omega \in \Omega$. Then
\begin{eqnarray*}
j_X(\omega, B) & = & \int_{{\mathbb{R}}^{+} \times {\mathbb{R}}} \ind_{B}(s, x) j_X(\omega, d(s,x))\\
& = & \sum_{n=1}^\infty \ind_B\big(T_n(\omega), \Delta X_{T_n(\omega)}(\omega)\big)\\
& = & \#\big\{n \in {\mathbb{N}} : \big(T_n(\omega), \Delta X_{T_n(\omega)}(\omega)\big) \in B \big\}\,.
\end{eqnarray*}
\end{corollary}
\begin{proof}
\smartqed
Since $\ind_{\{\Delta X \not= 0\}}(\omega, s) = \sum_{n=1}^\infty \ind_{\,\csbl T_n \csbr}(\omega, s) = \sum_{n=1}^\infty \ind_{\{T_n(\omega)\}}(s)$, we just have to permute the two sums.
\qed 
\end{proof}
}
Keeping the above representation of the jump measure $j_X$ in mind, we now are going to consider an important special case of a Borel set $B$ on ${\mathbb{R}}^{+} \times {\mathbb{R}}$, leading to the construction of ``stochastic'' integrals with respect to the jump measure $j_X$ including the construction of stochastic jump processes which play a fundamental role in the theory and application of L\'{e}vy processes.
{}
To this end, let us consider all Borel sets $B$ on ${\mathbb{R}}^{+} \times {\mathbb{R}}$ of type $B = [0, t] \times A$, where $t \geq 0$ and
\[
A \in {\mathcal{B}}^\ast : = \{A : A \in {\mathcal{B}}({\mathbb{R}}), 0 \notin \overline{A}\}\,.
\]
Obviously, $A \subseteq {\mathbb{R}}\setminus(-\varepsilon, \varepsilon)$ for \textit{all} $\varepsilon > 0$, implying in particular that $A \in {\mathcal{B}}^\ast$ is bounded from below. Let us recall the following
\begin{lemma}\label{thm:Applebaum}
Let $X = (X_t)_{t\geq 0}$ be a c\`{a}dl\`{a}g process. Let $A \in {\mathcal{B}}^\ast$ and $t > 0$. Then $N_X^A(t) : = j_X(\cdot, [0, t] \times A) < \infty$ a.\,s. 
\end{lemma}
\begin{proof}
This is \cite[Lemma 2.3.4.]{A2009}.
\qed
\end{proof}
\begin{prop}\label{thm:integral}
Let $X = (X_t)_{t\geq 0}$ be a c\`{a}dl\`{a}g process and $f : {\mathbb{R}}^{+} \times {\mathbb{R}} \rightarrow {\mathbb{R}}$ be measurable. Let $A \in {\mathcal{B}}^\ast$ and $t > 0$. Then for all $\omega \in \Omega$ the function $\ind_{[0, t] \times A} \, f$ is a.\,s. integrable with respect to the jump measure $j_X(\omega, d(s,x))$, and 
\begin{eqnarray*}
& {} & \int\limits_{[0, t] \times A} f(s, x) j_X(\omega, d(s,x)) {}\\
& =  & \sum_{0 < s \leq t} f\big(s, \Delta X_s(\omega)\big)\ind_{A}(\Delta X_s(\omega)){}\\
& = & \sum_{n=1}^\infty f\big(T_n(\omega), \Delta X_{T_n(\omega)}(\omega)\big)\ind_{A}\big(\Delta X_{T_n(\omega)}\big)\ind_{\{T_n \leq t\}}(\omega). 
\end{eqnarray*}
{}
Moreover, given $\omega \in \Omega$ there exists $c_t^A(\omega) \in \R^{+}$ such that 
\[
\int\limits_{[0, t] \times A} \vert f(s, x) \vert j_X(\omega, d(s,x)) \leq c_t^A(\omega) \, j_X(\omega, [0, t] \times A)\,.
\]
\end{prop}
\begin{proof}
Fix $\omega \in \Omega$ and consider the measurable function $g_t^A : = \ind_{[0, t] \times A}\, f$. Then ${\mathbb{R}}^{+} \times {\mathbb{R}} = B_1(\omega) \dotcup B_2(\omega)$, where $B_1(\omega) : = \{(s, \Delta X_s(\omega) : s > 0\}$ and $B_2(\omega) : = {\mathbb{R}}^{+} \times {\mathbb{R}}\setminus B_1(\omega)$. Obviously, we have 
\[
j_X(\omega, B_2(\omega)) = \sum\limits_{s > 0} \ind_{B_2(\omega)}\big(s, \Delta X_s(\omega)\big)\ind_{{\mathbb{R}}^\ast}(\Delta X_s(\omega)) = 0\,, 
\]
implying that $I_2 : = \int\limits_{B_2(\omega)} \vert g_t^A(s, x) \vert j_X(\omega, d(s,x)) = 0$. Put $I_1 : = \int\limits_{B_1(\omega)} \vert g_t^A(s, x) \vert j_X(\omega, d(s,x))$. Since on $[0, t]$ the c\`{a}dl\`{a}g path $s \mapsto X_s(\omega)$ has only finitely many jumps in $A \in {\mathcal{B}}^\ast$ there exist finitely many elements $(s_1, \Delta X_{s_1}(\omega)), \ldots, (s_N, \Delta X_{s_N}(\omega))$ which all are elements of $\big([0, t] \times A\big) \cap B_1(\omega)$ (for some $N = N(\omega, t, A) \in {\mathbb{N}}$). Put 
\[
0 \leq c_t^A(\omega) : = \max\limits_{1 \leq k \leq N}{}\vert f(s_k, \Delta X_{s_k}(\omega))\vert < \infty \,. 
\]
Then
\[
\vert g_t^A \vert = \ind_{[0, t] \times A} \, \vert f \vert \leq  c_t^A(\omega) \, \ind_{[0, t] \times A} \mbox{ on } B_1\,, 
\] 
and it follows that $I_2 \leq c_t^A(\omega) \, j_X(\omega, [0,t] \times A)$. A standard monotone class argument finishes the proof.\qed
\end{proof}
{}
\begin{rem}
Note that in terms of the previously discussed stopping times $S_n^A$ we may write   
\[
\int\limits_{[0, t] \times A} f(s, x) j_X(\omega, d(s,x)) = \sum_{n=1}^\infty f\big(S_n^A(\omega), \Delta X_{S_n^A(\omega)}(\omega)\big)\ind_{\{S_n^A \leq t\}}(\omega)\,.
\]
\end{rem}
{}
In the case of a L\'{e}vy process $X$ the following important special cases $f(s,x) : = 1$ and $f(s,x) : = x$ are embedded in the following crucial result (cf. e.\,g. \cite{A2009}):
\begin{thm}
Let $X = (X_t)_{t\geq 0}$ be a (c\`{a}dl\`{a}g) L\'{e}vy process and $A \in {\mathcal{B}}^\ast$.
\begin{description}[(ii)]
\item[(i)] Given $t \geq 0$ 
\begin{eqnarray*}
N_X^A(t) & = & \int\limits_{A} N_X^{dx}(t) := j_X(\cdot, [0, t] \times A) = \int\limits_{[0, t] \times A} j_X(\cdot, d(s,x)) \\
& \,\,=  & \sum_{0 < s \leq t} \ind_{A}(\Delta X_s) \, = \, \sum_{n=1}^\infty \ind_{A}\big(\Delta X_{T_n}\big)\ind_{\{T_n \leq t\}} \, = \, \sum_{n=1}^\infty \ind_{\{S_n^A \leq t\}} 
\end{eqnarray*}
induces a Poisson process $N_X^A = \big(N_X^A(t)\big)_{t \geq 0}$ with intensity measure $\nu_X(A) : = {\mathbb{E}}[N_X^A(1)] < \infty$.
\item[(ii)] Given $t \geq 0$ and a Borel measurable function $g : {\mathbb{R}} \longrightarrow {\mathbb{R}}$
\begin{eqnarray*}
Z_X^A(t) & : = & \int\limits_{A} g(x)\, N_X^{dx}(t) = \int\limits_{[0, t] \times A} g(x)\, j_X(\cdot, d(s,x))\\
& \,\, = & \sum_{0 < s \leq t} g\big(\Delta X_s\big) \, \ind_{A}(\Delta X_s) = \sum_{n=1}^\infty g\big(\Delta X_{T_n}\big)\ind_{A}\big(\Delta X_{T_n}\big)\ind_{\{T_n \leq t\}}\\
& \,\, = & \sum_{n=1}^\infty g\big(\Delta X_{S_n^A}\big)\ind_{\{S_n^A \leq t\}} = \sum_{n = 1}^{N_X^A(t)}g\big(\Delta X_{S_n^A}\big) 
\end{eqnarray*}
induces a compound Poisson process $Z_X^A = \big(Z_X^A(t)\big)_{t \geq 0}$. Moreover, if $g \in L^1(A, \nu_X)$ then ${\mathbb{E}}[Z_X^A(t)] = t\nu_X(A){\mathbb{E}}[g\big(\Delta X_{S_1^A}\big)]$.
\end{description}
\end{thm}
\section{Jump Measures of Optional Processes with Regulated Trajectories}
\label{sec:3}
One of the aims of our paper is to transfer particularly Theorem \ref{thm:Dellacherie} to the class of optional processes with regulated trajectories in order to construct a well-defined jump measure of such optional processes.

As we have seen the right-continuity of the paths of $X$ plays a significant role in the proof of Proposition \ref{thm:iterative exhausting representation}. We will see that a similar result holds for optional processes with regulated trajectories. However, it seems that we cannot simply implement the above sequence $(S_n^A)_{n \in \N}$ if the paths of $X$ are not right-continuous.

Our next contribution shows that we are not working with ``abstract nonsense'' only:
\begin{ex}
Optional processes which do not necessarily have right-continuous paths have emerged as naturally appearing candidates in the framework of enlarged filtration in financial mathematics (formally either describing ``insider trading information'' or ``extended information by inclusion of the default time of a counterparty'') including the investigation of the problem whether the no-arbitrage conditions are stable with respect to a progressive enlargement of filtration and how an arbitrage-free semimartingale model is affected when stopped at a random horizon (cf. \cite{ACDJ_1_2013}, \cite{ACDJ_2_2013} and \cite{ACDJ_3_2013}). 

Given a random time $\tau$, one can construct the smallest right-continuous filtration $\mathbb{G}$ which contains the given filtration $\mathbb{F}$ and makes $\tau$ a $\mathbb{G}$-stopping time (known as \textit{progressive enlargement of $\mathbb{F}$ with $\tau$}). Then one can associate to $\tau$ the two $\mathbb{F}$-supermartingales $Z$ and $\widetilde{Z}$, defined through
\[
Z_t : = {\mathbb{P}}(\tau > t \vert {\mathcal{F}}_t ) \textup { and } {\widetilde{Z}}_t : = {\mathbb{P}}(\tau \geq t \vert {\mathcal{F}}_t )\,.
\]   
$Z$ is c\`{a}dl\`{a}g, while $\widetilde{Z}$ is an optional process with regulated trajectories only.
\end{ex}
A first step towards the construction of a similar iterative and implementable exhausting representation of the set $\{\Delta X \not= 0\}$ for optional processes is encoded in the following
\begin{prop}\label{thm:finitely_layered_jump_rep}
Let $f : {\R^{+}} \longrightarrow {\mathbb{R}}$ be an arbitrary regulated function. Then 
\[
J_f = \dju{n=1}{\infty}{3.3}{2.0}D_n\,,
\] 
where each $D_n$ is a finite set.
\end{prop}
\begin{proof}
\smartqed
Since $(0, \infty) = \dju{n=1}{\infty}{7.0}{6.0}(n-1, n]$ it follows that $J_f =  \dju{n=1}{\infty}{7.0}{6.0}J_{f_n}$, where $f_n : = f\vert_{(n-1, n]}$ denotes the restriction of $f$ to the interval $(n-1, n]$. Fix $n \in \N$. Since every bounded infinite set of real numbers has a limit point (by Bolzano-Weierstrass) the at most countable set 
\[
J_{f_n}(\frac{1}{m}) = \big\{t: n-1 < t \leq n \textup{ and } \vert \Delta f(t) \vert > \frac{1}{m}\big\}
\]
must be already finite for each $m \in \N$ (cf. \cite[Theorem 2.6]{B1998} and \cite[p. 286-288]{Ho1921}). Moreover, $J_{f_n}(\frac{1}{m}) \subseteq J_{f_n}(\frac{1}{m+1})$ for all $m \in \N$. Consequently, we have
\[
J_{f_n} = \bigcup_{m = 1}^{\infty} J_{f_n}(\frac{1}{m}) = \dju{m=1}{\infty}{3.4}{1.8}A_{m,n},
\]
where $A_{1, n} : = J_{f_n}(1) = \{\vert \Delta f_n \vert > 1\}$ and $A_{m+1, n} := \{\Delta f_n \in \big(\frac{1}{m+1}, \frac{1}{m}\big]\}$ for all $m \in \N$, and hence
\[
J_f = \dju{n=1}{\infty}{3.4}{1.8}\dju{m=1}{\infty}{3.4}{1.8}A_{m,n}\,.
\]
Since $A_{m, n} \subseteq J_{f_n}(\frac{1}{m})$ for all $m \in \N$, each set $A_{m,n}$ consists of finitely many elements only.
\qed
\end{proof}
\begin{lemma}\label{lemma:recursive construction of finite sets}
Let $\emptyset \not= D$ be a finite subset of $\R$, consisting of
$\kappa_D$ elements. Consider
\[
s_1^D : = \min(D)
\]
and, if $\kappa_D \geq 2$,
\[
s_{n}^D := \min(D \cap (s_{n-1}^D, \infty)\big) = \min\{t > s_{n-1}^D: t \in D\},
\]
where $n \in \{2, 3, \ldots, \kappa_D\}$. Then $D \cap (s_{n-1}^D, \infty) \not= \emptyset$ and $s_{n-1}^D < s_{n}^D$ for all $n \in \{2, 3, \ldots, \kappa_D\}$. Moreover, we have
\[
D = \big\{s_1^D, s_2^D, \ldots, s_{\kappa_D}^D\big\}\,.
\]
\end{lemma}
\begin{proof}
\smartqed
Let $\kappa_D \geq 2$. Obviously, it follows that $D \cap (s_{1}^D, \infty) \not= \emptyset$. Now assume by contradiction that there exists $n \in \{2, \ldots, \kappa_D - 1\}$ such that $D \cap (s_{n}^D, \infty) = \emptyset$. Let $m^\ast$ be the smallest $m \in \{2, \ldots, \kappa_D - 1\}$ such that $D \cap (s_{m}^D, \infty) = \emptyset$. Then $s_{k}^D : = \min(D \cap (s_{k-1}^D, \infty)\big)
\in D$ is well-defined for all $k \in \{2, \ldots, m^\ast\}$, and we obviously have $s_{1}^D < s_{2}^D < \ldots < s_{m^\ast}^D$. Moreover, by construction of $m^\ast$, it follows that
\begin{equation}\label{eq:induction statement}
s \leq s_{m^\ast}^D \textup{ for all } s \in D.
\end{equation}
Assume now that there exists $\widetilde{s} \in D$ such that $\widetilde{s} \not\in \{s_{1}^D, s_{2}^D, \ldots, s_{m^\ast}^D\}$. Then, by (\ref{eq:induction statement}), there must exist $l \in \{1, 2, \ldots, m^\ast-1 \}$ such that $s_{l}^D < \widetilde{s} < s_{l+1}^D$, which is a contradiction, due to the definition of $s_{l+1}^D$. Hence, $\widetilde{s}$ cannot exist, and it consequently follows that $D = \{s_{1}^D, s_{2}^D, \ldots, s_{m^\ast}^D\}$. But then $m^\ast = \#(D) \leq \kappa_D - 1 < \kappa_D$, which is a contradiction. Hence, $D \cap (s_{n}^D, \infty) \not= \emptyset$ for any $n \in \{2, \ldots, \kappa_D - 1\}$, implying that $s_{n}^D \in D$ is well-defined and $s_{n}^D < s_{n+1}^D$ for all $n \in \{1, 2, \ldots, \kappa_D - 1\}$. Clearly, we must have $D = \big\{s_1^D, s_2^D, \ldots, s_{\kappa_D}^D\big\}$.
\qed
\end{proof}
Let $A \subseteq \Omega \times {\mathbb{R}}^{+}$ and $\omega \in
\Omega$. Consider
\[
D_A(\omega) : = \inf \{t \in  {\mathbb{R}}^{+} : (\omega , t) \in A\}
\in [0, \infty]
\]
$D_A$ is said to be the \textit{d\'{e}but} of $A$. Recall that
$\inf(\emptyset) = + \infty$ by convention. $A$ is called a
\textit{progressive set} if $\ind_{A}$ is a progressively measurable
process. For a better understanding of the main ideas in the proof
of Theorem \ref{thm:jumps of optional processes and stopping
times}, we need the following non-trivial result (a detailed proof
of this statement can be found in e.\,g. \cite{HWY1992}):
\begin{theorem}\label{thm:optional debut is stopping time}
Let $A \subseteq \Omega \times {\mathbb{R}}^{+}$. If $A$ is
a progressive set, then $D_A$ is a stopping time.
\end{theorem}
Next, we reveal how these results enable a transfer of the jump measure for c\`{a}dl\`{a}g and adapted processes to optional processes with infinitely many jumps and regulated trajectories which need not necessarily be right-continuous. To this end, we firstly generalise Theorem \ref{thm:Dellacherie} in the following sense:
\begin{theorem}\label{thm:jumps of optional processes and stopping times}
Let $X : \Omega \times \R^+ \longrightarrow \R$ be an
optional process such that all trajectories of $X$ are regulated 
and $\Delta X_0 = 0$. Then $\Delta X$ is also optional. 
If for each trajectory of $X$ its set of jumps is not finite, then there
exists a sequence of stopping times $(T_n)_{n \in \N}$ such that
$(T_n(\omega))_{n \in \N}$ is a strictly increasing sequence in $(0,
\infty)$ for all $\omega \in \Omega$ and
\[
J_{X_{\bullet}(\omega)} = \dju{n=1}{\infty}{3.3}{2.4}\{T_n(\omega)\} \textup{
for all } \omega \in \Omega,
\]
or equivalently,
\[
\{\Delta X \not= 0\} = \dju{n = 1}{\infty}{3.3}{2.4}\rsi T_n \lsi\,.
\]
In particular $\{\Delta X \not= 0\}$ is a thin set.
\end{theorem}
\begin{proof}
\smartqed
Due to the assumption on $X$ and Lemma \ref{lemma:optional and regulated paths}, 
$X^{-}$ is predictable, $X^{+}$ is adapted
and all trajectories of $X^{+}$ are
right-continuous on $\R^+$. Hence, by Theorem \ref{thm:POPA} both,
$X^{-}$ and $X^+$ are optional processes, implying that
the jump process $\Delta X = X^{+} - X^{-}$ is optional as well.

Fix $\omega \in \Omega$. Consider the trajectory $f : =
X_{\bullet}(\omega)$. Due to Proposition 
\ref{thm:finitely_layered_jump_rep} we may represent $J_f$ as
\[
J_f = \dju{m = 1}{\infty}{3.4}{2.2} D_m(\omega),
\]
where $\kappa_m(\omega) : = {\#}(D_m(\omega)) < +\infty$
for all $m \in \N$.
Let ${\mathbb{M(\omega)}} : =\{m \in \N : D_{m}(\omega) \not=
\emptyset\}$. Fix an arbitrary $m \in {\mathbb{M(\omega)}}$. Consider
\[
0 < S_1^{(m)}(\omega) : = \min(D_{m}(\omega))
\]
and, if $\kappa_m(\omega) \geq 2$,
\[
0 < S_{n+1}^{(m)}(\omega) := \min\big(D_{m}(\omega) \cap
(S_n^{(m)}(\omega), \infty)\big),
\]
where $n \in \{1, 2, \ldots, \kappa_m(\omega) - 1\}$.
Since $\Delta X$ is optional, it follows that $\{\Delta X \in B\}$
is optional for all Borel sets $B \in \mathcal{B}(\R)$. Moreover,
since $\Delta f(0) = \Delta X_0(\omega) : = 0$ (by assumption), it
actually follows that $\{s \in \R^+ : (\omega, s) \in \{\Delta X \in
C\}\} = \{s \in (0, \infty) : (\omega, s) \in \{\Delta X \in C\}\}$
for all Borel sets $C \in \mathcal{B}(\R)$ which do not contain $0$.
Hence, as the construction of the sets $D_{m}(\omega)$ in the proof
of Proposition \ref{thm:finitely_layered_jump_rep} clearly
shows, $S_1^{(m)}$ is the d\'{e}but of an optional set.
Consequently, due to Theorem \ref{thm:optional debut is stopping
time}, it follows that $S_1^{(m)}$ is a stopping time. If
$S_n^{(m)}$ is a stopping time, the stochastic interval $\lsi
S_n^{(m)}, \infty \rsi$ is optional too (cf. \cite{HWY1992},
Theorem 3.16). Thus, by construction, $S_{n+1}^{(m)}$ is the
d\'{e}but of an optional set and hence a stopping time. Due to Lemma
\ref{lemma:recursive construction of finite sets}, we have
\[
J_f = \dju{m \in {\mathbb{M(\omega)}}}{}{6.0}{2.4} D_{m}(\omega) =
\dju{m \in {\mathbb{M(\omega)}}}{}{6.0}{4.0} \dju{n =
1}{\kappa_m(\omega)}{4.5}{3.0}\{S_n^{(m)}(\omega)\}.
\]
Hence, since for each trajectory of $X$ its set of jumps is not
finite, the at most countable set ${\mathbb{M(\omega)}}$ is not
finite, hence countable, and a simple relabeling of the stopping
times $S_n^{(m)}$ finishes the proof.
\qed
\end{proof}
\begin{theorem}\label{thm:optional random measures}
Let $X : \Omega \times \R^+ \longrightarrow \R$ be an
optional process such that all trajectories of $X$ are regulated, 
$\Delta X_0 = 0$ and the set of jumps of each trajectory of $X$ is
not finite. Then the function
\begin{eqnarray*}
j_X : \Omega \times \mathcal{B}(\R^+) \otimes
\mathcal{B}(\R) & \longrightarrow & \Z^+ \cup \{+ \infty\}\\
(\omega, G) & \mapsto & \sum_{s > 0} \ind_{G}\big(s, \Delta
X_{s}(\omega) \big) \ind_{\{\Delta X \not= 0\}}(\omega, s)
\end{eqnarray*}
is an integer-valued random measure.
\end{theorem}
\begin{proof}
\smartqed
We only have to combine Theorem \ref{thm:jumps of optional
processes and stopping times} and \cite{HWY1992}, Theorem
11.13.
\qed
\end{proof}
Implementing the exhausting series of stopping times $(T_n)_{n \in \N}$ of the thin set $\{\Delta X \not= 0\}$ from Theorem \ref{thm:jumps of optional processes and stopping times}, we immediately obtain 
\begin{corollary}
Let $B \in {\mathcal{B}}{}({\mathbb{R}}^{+} \times {\mathbb{R}}{})$ and $\omega \in \Omega$. Then
\begin{eqnarray*}
j_X(\omega, B) & = & \int_{{\mathbb{R}}^{+} \times {\mathbb{R}}} \ind_{B}(s, x) j_X(\omega, d(s,x))\\
& = & \sum_{n=1}^\infty \ind_B\big(T_n(\omega), \Delta X_{T_n(\omega)}(\omega)\big)\\
& = & \#\big\{n \in {\mathbb{N}} : \big(T_n(\omega), \Delta X_{T_n(\omega)}(\omega)\big) \in B \big\}\,.
\end{eqnarray*}
\end{corollary}
\begin{proof}
\smartqed
Since $\ind_{\{\Delta X \not= 0\}}(\omega, s) = \sum_{n=1}^\infty \ind_{\,\csbl T_n \csbr}(\omega, s) = \sum_{n=1}^\infty \ind_{\{T_n(\omega)\}}(s)$, we just have to permute the two sums.
\qed 
\end{proof}
We finish our paper with the following two natural questions:
\begin{prb}
Let $X : \Omega \times \R^+ \longrightarrow \R$ be an
optional process such that all trajectories of $X$ are regulated 
and $\Delta X_0 = 0$. Does Lemma \ref{thm:Applebaum} hold for $X$?
\end{prb}
\begin{prb}
Let $X : \Omega \times \R^+ \longrightarrow \R$ be an
optional process such that all trajectories of $X$ are regulated 
and $\Delta X_0 = 0$. Does Proposition \ref{thm:integral} hold for $X$?
\end{prb}
\begin{acknowledgement}
The author would like to thank Monique Jeanblanc for the indication of the very valuable references \cite{ACDJ_1_2013}, \cite{ACDJ_2_2013} and \cite{ACDJ_3_2013}.   
\end{acknowledgement}


\begin{thebibliography}{99.}%
%
%
%
%
%
%
%

%
%
\bibitem{ACDJ_1_2013} A. Aksamit, T. Choulli, J. Deng, M. Jeanblanc, \textit{Arbitrages in a Progressive Enlargement 
Setting}. (Preprint: http://arxiv.org/abs/1312.2433, 2013)

\bibitem{ACDJ_2_2013} A. Aksamit, T. Choulli, J. Deng, M. Jeanblanc, \textit{Non-Arbitrage up to Random Horizons and
after Honest Times for Semimartingale Models}. (Preprint: http://arxiv.org/abs/1310.1142v2, 2014)

\bibitem{ACDJ_3_2013} A. Aksamit, T. Choulli, J. Deng, M. Jeanblanc, \textit{Non-Arbitrage Under Additional Information for Thin Semimartingale Models}. (Preprint: http://arxiv.org/abs/1505.00997v1, 2015)

\bibitem{A2009} D. Applebaum, \textit{L\'{e}vy Processes and Stochastic Calculus}, 2nd edn. (Cambridge University Press, 2009)

\bibitem{B1998} F. Burk, \textit{Lebesgue Measure and Integration. An Introduction.}. (John Wiley \& Sons, New York, 1998)

\bibitem{CW1990} K. L. Chung, R. J. Williams, \textit{Introduction to Stochastic Integration}, 2nd edn. (Birkh\"{a}user Boston 1990)

\bibitem{CT2004} R. Cont, P. Tankov, \textit{Financial Modelling with Jump Processes}. (Chapman \& Hall, 2004)

\bibitem{D1972} C. Dellacherie, \textit{Capacit\'{e}s et processus stochastiques}. (Springer, 1972)

\bibitem{D1960} J. Dieudonn\'{e}, \textit{Foundations of Modern Analysis}. (Pure and Applied Mathematics, Academic Press, New York and London, 1960)
\bibitem{HWY1992} S. He, J. Wang, J. Yan, \textit{Semimartingale Theory and Stochastic Calculus}. (Science Press and CRC Press, Inc., 1992)

\bibitem{Ho1921} E. W. Hobson, \textit{The Theory of Functions of a Real Variable and the Theory of Fourier Series}. (Cambridge University Press, 1921)

\bibitem{JS2003} J. Jacod, A. Shiryaev, \textit{Limit Theorems for Stochastic Processes}, 2nd edn. (Springer, 2003)

\bibitem{K2004} F. C. Klebaner, \textit{Introduction to Stochastic Calculus with Applications}, 2nd edn. (Imperial College Press, 2004)

\bibitem{P2004} P. Protter, \textit{Stochastic Integration and Differential Equations}, 2nd edn. (Springer, 2004)

%
%
%
%
%
\end{thebibliography}
\end{document}